\documentclass[12pt]{article}
\usepackage{amsmath, amssymb, amsfonts}
\usepackage{amsthm}
\usepackage[utf8]{inputenc}
\usepackage{csquotes}
\usepackage{geometry}
\usepackage{adjustbox}
\usepackage{picture}
\usepackage{graphicx}
\usepackage{verbatim}
\usepackage{array}
\usepackage{enumitem}
\usepackage{booktabs}
\usepackage{hyperref}
\usepackage{url}
\usepackage[capitalize]{cleveref}
\usepackage{tikz}

\title{Axiomatic Foundations of Fractal Analysis and Fractal Number Theory}
\author{Stanislav Semenov \\
\href{mailto:stas.semenov@gmail.com}{stas.semenov@gmail.com} \\
\href{https://orcid.org/0000-0002-5891-8119}{ORCID: 0000-0002-5891-8119}}
\date{April 2, 2025}

\theoremstyle{definition}
\newtheorem{definition}{Definition}[section]

\newtheorem{example}{Example}[section]
\newtheorem{conjecture}[definition]{Conjecture}

\theoremstyle{plain}
\newtheorem{axiom}{Axiom}[section]
\newtheorem{theorem}[definition]{Theorem}

\newtheorem{corollary}[definition]{Corollary}

\newtheorem{principle}{Principle}[section]

\theoremstyle{remark}
\newtheorem*{remark}{Remark}

\begin{document}

\maketitle

\begin{abstract}
We develop an axiomatic framework for \emph{fractal analysis} and \emph{fractal number theory} grounded in hierarchies of definability. Central to this approach is a sequence of formal systems \( \mathcal{F}_n \), each corresponding to a definability level \( S_n \subseteq \mathbb{R} \) of constructively accessible mathematical objects. This structure refines classical analysis by replacing uncountable global constructs with countable, syntactically constrained approximations.

The axioms formalize:
\begin{itemize}
    \item A hierarchy of definability levels \( S_n \), indexed by syntactic and ordinal complexity;
    \item Fractal topologies and the induced notions of continuity, compactness, and differentiability;
    \item Layered integration and differentiation with explicit convergence and definability bounds;
    \item Arithmetic and function spaces over the stratified continuum \( \mathbb{R}_{S_n} \subseteq \mathbb{R} \).
\end{itemize}

This framework synthesizes constructive mathematics, proof-theoretic stratification, and fractal geometric intuition into a unified, finitistically structured model. Key results include the definability-based classification of real numbers (e.g., algebraic, computable, Liouville), a stratified fundamental theorem of calculus with syntactic error bounds, and compatibility with base systems such as \( \mathsf{RCA}_0 \) and \( \mathsf{ACA}_0 \). 

The framework enables constructive approximation and syntactic regularization of classical analysis, with applications to proof assistants, computable mathematics, and foundational studies of the continuum.
\end{abstract}

\subsection*{Mathematics Subject Classification}
03F60 (Constructive and recursive analysis), 26E40 (Constructive analysis), 03F03 (Proof theory and constructive mathematics)

\subsection*{ACM Classification}
F.4.1 Mathematical Logic, F.1.1 Models of Computation

\section*{Introduction}
\label{sec:intro}

The foundational crisis of the early 20th century revealed fundamental tensions between the principal interpretations of the mathematical continuum:

\begin{itemize}
\item The \textit{classical continuum}, as formalized in Zermelo–Fraenkel set theory with the Axiom of Choice (ZFC), which permits non-constructive existence proofs and impredicative definitions.
\item \textit{Constructive approaches}, such as Brouwer’s intuitionism and Bishop’s constructive analysis, which reject the law of excluded middle and require all mathematical objects to have computational content.
\item \textit{Computational realizability frameworks}, such as Turing machines and Type-2 effectivity (TTE), which seek to ground mathematical structures in algorithmic representations.
\end{itemize}

Despite significant advances in each direction, there remains no unified framework that simultaneously preserves constructive rigor, syntactic tractability, and topological expressiveness. This paper introduces such a framework via the concept of \emph{fractal definability}—a hierarchical model of real numbers and analysis stratified by syntactic complexity and proof-theoretic strength.

\paragraph{Core Contributions}

We propose a countably stratified continuum \( \mathbb{R}_{S_\omega} = \bigcup_{n \in \mathbb{N}} \mathbb{R}_{S_n} \), in which each layer \( \mathbb{R}_{S_n} \) represents the class of real numbers constructively definable at a bounded level of logical and computational complexity. The resulting system offers:

\begin{enumerate}
\item A constructive continuum with fine-grained control over definability, extending beyond monolithic models such as the computable reals.
\item Layer-relative versions of analytical concepts—continuity, compactness, differentiability, integration—defined relative to \( S_n \)-topologies and approximation mechanisms.
\item A syntactic arithmetic hierarchy of real numbers that explicitly distinguishes algebraic, computable transcendental, and non-computable reals through definitional complexity.
\end{enumerate}

This framework is formalized axiomatically, with distinct systems governing:

\begin{itemize}
\item \textbf{Fractal definability}: axioms characterizing definable sets, functions, and reals at level \( S_n \).
\item \textbf{Fractal topology}: open sets and limit processes defined over \( S_n \)-intervals.
\item \textbf{Fractal arithmetic and calculus}: field operations, derivatives, and integrals defined within stratified layers.
\item \textbf{Fractal number theory}: a structured classification of real numbers based on their definability properties.
\end{itemize}

\paragraph{Philosophical Motivation}

Constructivity is not inherently binary; rather, it admits degrees of definitional power. Our framework models the continuum not as a single undivided totality but as a union of stratified layers of constructive accessibility. Each level reflects a bounded syntactic universe, giving rise to a \textit{fractal continuum}—a constructively grounded analogue of the classical real line, nuanced by definability bounds.

\paragraph{Related Work}

\begin{itemize}
    \item \textbf{Bishop's constructive analysis} \cite{Bishop1967,Bridges1986} provides a minimal classical-free foundation for analysis, but treats definability uniformly, without internal complexity gradation.
    \item \textbf{Reverse mathematics} \cite{Simpson2009} calibrates the logical strength of theorems via subsystems like \( \mathsf{RCA}_0 \), \( \mathsf{ACA}_0 \), etc., yet focuses on theorems rather than definability of individual objects.
    \item \textbf{Computable analysis} \cite{Weihrauch2000} formalizes computation over real numbers via effective procedures, but lacks a syntactic hierarchy of constructive depth.
    \item \textbf{Descriptive set theory} offers pointclass hierarchies and topological classifications, though often within classical logic and with no intrinsic resource bounds.
\end{itemize}

The present system synthesizes these traditions by embedding definability constraints directly into the structure of real analysis, yielding a layered continuum compatible with both constructive reasoning and computational analysis.

This framework extends ideas introduced in earlier work on fractal countability and constructive alternatives to classical set-theoretic hierarchies \cite{Semenov2025FractalAnalysis, Semenov2025FractalBoundaries, Semenov2025FractalCountability}.

\section{Preliminaries}
\label{sec:prelim}

We assume standard familiarity with the following foundational topics:

\begin{itemize}
    \item Ordinal notations and the arithmetical hierarchy, as developed in classical recursion theory \cite{SoareRecursion}
    \item Constructive real analysis in the sense of Bishop and Bridges \cite{Bishop1967,Bridges1986}
    \item Proof-theoretic subsystems of second-order arithmetic, such as $\mathsf{RCA}_0$, $\mathsf{ACA}_0$, and related frameworks in reverse mathematics \cite{Simpson2009}
\end{itemize}

Throughout this paper, we adopt a stratified framework in which definability is layered by a hierarchy of formal systems \( \mathcal{F}_n \), indexed by natural numbers \( n \in \mathbb{N} \). Each level \( \mathcal{F}_n \) governs a corresponding class \( S_n \) of constructively definable mathematical objects, bounded in complexity by a parameterized resource function.

\subsection*{Key Notations and Definitions}

\begin{itemize}
    \item \( \mathcal{F}_n \): A formal system at level \( n \), allowing definitions of functions and sets using bounded syntactic complexity. The class of functions definable in \( \mathcal{F}_n \) is assumed to be total and computable within a specified resource bound \( \rho_n(k) \), where \( k \) is the input size and \( \rho_n \) is a tunable parameter (e.g., \( \rho_n(k) = \exp_n(k) \), \( k! \), etc.).

    \item \( S_n \): The definability level associated with \( \mathcal{F}_n \), containing all objects (natural numbers, rationals, functions, sets) that admit constructive definitions in \( \mathcal{F}_n \).

    \item \( \mathbb{Q}_{S_n} := \mathbb{Q} \cap S_n \): The set of rational numbers whose finite representations (as reduced fractions \( p/q \)) are definable within \( \mathcal{F}_n \).

    \item \( \mathbb{R}_{S_n} \): The set of \emph{fractal reals} at level \( n \), defined as real numbers that admit rapidly converging sequences from \( \mathbb{Q}_{S_n} \) with convergence governed by a definable modulus in \( \mathcal{F}_n \). This definition will be formalized in Section~\ref{sec:fractal-reals}.

    \item \( \mathcal{T}_n \): The $S_n$-induced topology on \( \mathbb{R} \), generated by basic open intervals with endpoints in \( \mathbb{Q}_{S_n} \). This topology governs continuity and compactness relative to definability constraints.
\end{itemize}

\section{Fractal Real Numbers}
\label{sec:fractal-reals}

We define the set of fractal reals at level \( n \) using a constructively stratified version of the Cauchy completion of rationals:

\begin{definition}[Fractal Reals]
\label{def:fractal-reals}
A real number \( x \in \mathbb{R} \) belongs to \( \mathbb{R}_{S_n} \) if and only if there exists a total function
\[
f \colon \mathbb{N} \to \mathbb{Q}_{S_n}
\]
such that:
\[
\forall k \in \mathbb{N},\quad |x - f(k)| < 2^{-k},
\]
and the function \( f \) is definable in the formal system \( \mathcal{F}_n \).
\end{definition}

This generalizes the Bishop-style definition of reals as rapidly converging rational sequences, by parameterizing the construction through definability within a bounded syntactic layer \( \mathcal{F}_n \). The result is a stratified family of constructive subfields of \( \mathbb{R} \), where definability replaces mere computability or intuitionistic provability.

\begin{remark}
The above definition is equivalent to taking the closure of \( \mathbb{Q}_{S_n} \) under the $S_n$-topology \( \mathcal{T}_n \), once the topology is defined via basic open intervals with \( S_n \)-definable endpoints. This duality between topological closure and syntactic convergence enables both algebraic and analytic treatments of \( \mathbb{R}_{S_n} \).
\end{remark}

\subsection*{Stratification and Constructive Closure}

The fractal real sets form an increasing sequence:
\[
\mathbb{R}_{S_0} \subseteq \mathbb{R}_{S_1} \subseteq \cdots \subseteq \mathbb{R}_{S_n} \subseteq \mathbb{R}_{S_{n+1}} \subseteq \cdots,
\]
with strict inclusion in general for \( n \geq 0 \). Their union defines the total constructive closure of the real continuum within this framework:
\[
\mathbb{R}_{S_\omega} := \bigcup_{n \in \mathbb{N}} \mathbb{R}_{S_n}.
\]

This stratified approach allows fine-grained control over computability, proof-theoretic strength, and syntactic complexity. It lays the foundation for a definability-sensitive treatment of analysis, number theory, and topology, enabling layered versions of classical theorems under restricted constructive resources.

\section{Axioms of Fractal Definability}
\label{sec:axioms-def}

\begin{axiom}[Initial Basis]
\label{ax:initial-basis}
The base formal system \( \mathcal{F}_0 \) is any finite syntactic core sufficient to define basic arithmetic expressions and derivations. The choice of \( \mathcal{F}_0 \) determines the initial layer \( S_0 \subset \mathbb{R}_{S_\omega} \).
\end{axiom}

\begin{axiom}[Conservative Hierarchy of Formal Systems]
\label{ax:hierarchy-formal-systems}
There exists an increasing sequence of formal systems \( \{ \mathcal{F}_n \}_{n \in \mathbb{N}} \) such that for each \( n \in \mathbb{N} \):
\begin{itemize}
    \item \( \mathcal{F}_0 \) is a fixed syntactic base (e.g., \( \mathsf{RCA}_0 \)) adequate for expressing primitive recursive arithmetic and basic constructions over \( \mathbb{Q} \);
    \item \( \mathcal{F}_{n+1} \) is a conservative syntactic extension of \( \mathcal{F}_n \), with strictly increased definitional power;
    \item Each \( \mathcal{F}_n \) proves all true \(\Sigma^0_n\)-statements provable in \( \mathcal{F}_{n-1} \) and admits uniform definitions for all total functions in the class \( \Delta^0_n \);
    \item The language of \( \mathcal{F}_n \) supports internal syntactic definitions of all objects in \( S_n \), including rational approximations, functions \( \mathbb{N} \to \mathbb{Q} \), and convergence witnesses;
    \item The proof-theoretic ordinal of \( \mathcal{F}_n \) does not exceed \( \omega^n \).
\end{itemize}
\end{axiom}

\begin{axiom}[Extended Hierarchy of Formal Systems]
\label{ax:extended-hierarchy}
There exists an increasing sequence of formal systems \( \{ \mathcal{F}_n \}_{n \in \mathbb{N}} \) satisfying the following:

\begin{itemize}
    \item \( \mathcal{F}_0 \) is a fixed syntactic base system (e.g., \( \mathsf{RCA}_0 \) or a finitely axiomatized fragment of arithmetic) sufficient to define basic arithmetic and rational constructions;
    
    \item For each \( n \), the system \( \mathcal{F}_{n+1} \) is a consistent syntactic extension of \( \mathcal{F}_n \), admitting strictly greater definitional power;
    
    \item Each \( \mathcal{F}_n \) supports:
    \begin{itemize}
        \item internal definitions of Cauchy sequences, rational functions, and convergence witnesses;
        \item uniform definitions for total functions in a syntactic class \( \Delta_n \subseteq \text{Def}_n \), where \( \text{Def}_n \) denotes the syntactic class of definitions available in \( \mathcal{F}_n \);
    \end{itemize}
    
    \item The proof-theoretic ordinal of \( \mathcal{F}_n \) is bounded by a function \( \alpha(n) \), where \( \alpha \colon \mathbb{N} \to \mathrm{Ord} \) is strictly increasing and may be unbounded in principle;
    
    \item The induced definability layers \( S_n \subseteq \mathbb{R} \) consist of all real numbers whose rational approximations and convergence proofs are expressible within \( \mathcal{F}_n \).
\end{itemize}
\end{axiom}

\begin{axiom}[Hierarchy of Definability]
\label{ax:hierarchy-definability}
For each natural number $n \in \mathbb{N}$, there exists a definability level $S_n$, consisting of all constructive objects and functions that are definable in finite time with ordinal complexity bounded by $n$. That is, $S_n$ contains precisely those objects whose definitions can be carried out by a formal system whose proof-theoretic ordinal does not exceed $\omega^n$.
\end{axiom}

\begin{axiom}[Inclusion and Closure]
\label{ax:inclusion}
For all $n \in \mathbb{N}$, the definability levels satisfy $S_n \subseteq S_{n+1}$, with strict inclusion for $n \geq 1$. The union over all levels forms the total constructive closure:
\[
S_\omega := \bigcup_{n \in \mathbb{N}} S_n.
\]
This set $S_\omega$ represents the full constructive universe: the closure of all finitely stratified definable objects under layer-wise approximation.
\end{axiom}

\begin{axiom}[Definability of Objects]
\label{ax:definability-objects}
A real number \( x \in \mathbb{R} \) belongs to the definability level \( S_n \) if and only if there exists a finite constructive specification of \( x \)---such as a computable function, formal term, or syntactic expression---that is derivable within the formal system \( \mathcal{F}_n \). That is, \( x \in S_n \) if and only if \( x \) is internally representable in \( \mathcal{F}_n \).
\end{axiom}

\begin{example}
The membership of a real number in some definability level $S_n$ depends on the choice of the base formal system $\mathcal{F}_0$. For example, if $\mathcal{F}_0 = \mathsf{RCA}_0$ \cite{Simpson2009}, then:
\begin{itemize}
    \item $\sqrt{2} \in S_1$, since it is an algebraic number definable via a polynomial with rational coefficients, and root-finding for such equations is available in $\mathsf{ACA}_0 \supset \mathcal{F}_1$;
    \item $e \in S_2$, as it is definable via a convergent Taylor expansion, whose convergence can be verified within a system allowing effective real analysis (e.g., $\mathsf{ACA}_0$ or a fragment with primitive limits);
    \item Chaitin's constant $\Omega \notin S_\omega$, because it is not computably approximable and encodes the halting probability of a universal Turing machine \cite{Chaitin1975}, requiring non-arithmetical comprehension.
\end{itemize}
Thus, the hierarchy $\{S_n\}$ is relative to the initial choice of $\mathcal{F}_0$ and grows by adding definability power step by step \cite{Semenov2025FractalBoundaries}.
\end{example}

\section{Axioms of Fractal Topology}
\label{sec:fractal-topology}

We define a stratified topological structure \( \mathcal{T}_n \) on each level \( \mathbb{R}_{S_n} \), induced by the definability bounds of \( S_n \), as proposed in \cite{Semenov2025FractalBoundaries,Semenov2025FractalCountability}.

\begin{axiom}[Definable Topology]
Each level \( S_n \) induces a topology \( \mathcal{T}_n \) on \( \mathbb{R}_{S_n} \), generated by a basis of open intervals \( (a, b) \), where:
\begin{itemize}
  \item \( a, b \in \mathbb{Q}_{S_n} \), \( a < b \),
  \item The mapping \( (a, b) \mapsto \{x \in \mathbb{R}_{S_n} \mid a < x < b\} \) is computable in \( \mathcal{F}_n \).
\end{itemize}
The resulting topology \( \mathcal{T}_n \) is second-countable and definably generated, with a countable basis effectively enumerable in \( \mathcal{F}_n \).
\end{axiom}

\begin{axiom}[Constructive Openness]
\label{ax:constructive-openness}
A set \( U \subseteq \mathbb{R}_{S_n} \) is open in the topology \( \mathcal{T}_n \) if and only if there exists a total \( \mathcal{F}_n \)-definable function
\[
f \colon \mathbb{N} \to \mathbb{Q}_{S_n} \times \mathbb{Q}_{S_n}
\]
such that for all \( k \in \mathbb{N} \), \( f(k) = (a_k, b_k) \) with \( a_k < b_k \), and:
\[
U = \bigcup_{k \in \mathbb{N}} (a_k, b_k),
\]
where the enumeration \( \{(a_k, b_k)\}_{k \in \mathbb{N}} \) is effective: that is,
\begin{itemize}
    \item Each pair \( (a_k, b_k) \) lies in \( \mathbb{Q}_{S_n} \times \mathbb{Q}_{S_n} \), with \( a_k < b_k \),
    \item The membership condition \( x \in U \) is semi-decidable relative to \( \mathcal{F}_n \): for any \( x \in \mathbb{R}_{S_n} \), there exists a computable search procedure (in \( \mathcal{F}_n \)) that finds \( k \) such that \( x \in (a_k, b_k) \),
    \item The convergence of the sequence \( \{(a_k, b_k)\} \) is not required; only that it covers \( U \).
\end{itemize}
\end{axiom}

\begin{definition}[Effectively Open Set in \( \mathcal{T}_n \)]
\label{def:effectively-open}
A set \( U \subseteq \mathbb{R}_{S_n} \) is said to be \emph{effectively open} at level \( n \) if there exists a total \( \mathcal{F}_n \)-definable function
\[
f \colon \mathbb{N} \to \mathbb{Q}_{S_n} \times \mathbb{Q}_{S_n}
\]
such that for each \( k \in \mathbb{N} \), \( f(k) = (a_k, b_k) \) with \( a_k < b_k \), and:
\[
U = \bigcup_{k \in \mathbb{N}} (a_k, b_k).
\]

Additionally, the enumeration is effective in the following sense:
\begin{itemize}
    \item The function \( f \) is total and its graph is computable in \( \mathcal{F}_n \).
    \item For any \( x \in \mathbb{R}_{S_n} \), the membership condition \( x \in U \) is semi-decidable within \( \mathcal{F}_n \), i.e., there exists a computable search procedure that finds \( k \) such that \( x \in (a_k, b_k) \).
\end{itemize}
\end{definition}

\begin{definition}[Effectively Closed Set in \( \mathcal{T}_n \)]
\label{def:effectively-closed}
A set \( C \subseteq \mathbb{R}_{S_n} \) is said to be \emph{effectively closed} at level \( n \) if its complement \( \mathbb{R}_{S_n} \setminus C \) is effectively open in \( \mathcal{T}_n \), i.e., there exists a total \( \mathcal{F}_n \)-definable function
\[
f \colon \mathbb{N} \to \mathbb{Q}_{S_n} \times \mathbb{Q}_{S_n}
\]
such that:
\[
\mathbb{R}_{S_n} \setminus C = \bigcup_{k \in \mathbb{N}} (a_k, b_k), \quad \text{where } f(k) = (a_k, b_k),\ a_k < b_k.
\]

Equivalently, a point \( x \in \mathbb{R}_{S_n} \) belongs to \( C \) if and only if:
\[
\forall k \in \mathbb{N},\ x \notin (a_k, b_k),
\]
and this verification can be carried out constructively in \( \mathcal{F}_n \).
\end{definition}

\begin{definition}[Effectively Compact Set in \( \mathcal{T}_n \)]
\label{def:effectively-compact}
A set \( K \subseteq \mathbb{R}_{S_n} \) is \emph{effectively compact} at level \( n \) if every effectively open cover of \( K \) admits a finite subcover that is computable in \( \mathcal{F}_n \).

More precisely: for every \( \mathcal{F}_n \)-definable function
\[
f \colon \mathbb{N} \to \mathbb{Q}_{S_n} \times \mathbb{Q}_{S_n}, \quad f(k) = (a_k, b_k),\ a_k < b_k,
\]
such that \( K \subseteq \bigcup_{k \in \mathbb{N}} (a_k, b_k) \), there exists a finite set \( \{k_1, \dots, k_m\} \subset \mathbb{N} \) computable in \( \mathcal{F}_n \), such that:
\[
K \subseteq \bigcup_{j = 1}^m (a_{k_j}, b_{k_j}).
\]
\end{definition}

\begin{axiom}[Fractal Continuity]
\label{ax:fractal-continuity}
A function \( f: \mathbb{R}_{S_n} \rightarrow \mathbb{R}_{S_n} \) is continuous at level \( S_n \) if for every open set \( U \in \mathcal{T}_n \), the preimage \( f^{-1}(U) \in \mathcal{T}_n \).
\end{axiom}

\begin{definition}[Effective Continuity on Effectively Compact Sets]
\label{def:effective-continuity-compact}
Let \( K \subseteq \mathbb{R}_{S_n} \) be an effectively compact set, and let \( f \colon K \to \mathbb{R}_{S_n} \) be a total function.

We say that \( f \) is \emph{effectively continuous on \( K \)} if there exists a total \( \mathcal{F}_n \)-definable function
\[
\delta \colon \mathbb{Q}_{S_n}^+ \to \mathbb{Q}_{S_n}^+
\]
such that for all \( \varepsilon \in \mathbb{Q}_{S_n}^+ \), and all \( x, y \in K \),
\[
|x - y| < \delta(\varepsilon) \ \Rightarrow\ |f(x) - f(y)| < \varepsilon.
\]

That is, the modulus of continuity \( \delta \) is computable in \( \mathcal{F}_n \), and uniformly controls the variation of \( f \) on \( K \).
\end{definition}

\begin{axiom}[Fractal Convergence]
\label{ax:fractal-convergence}
A sequence \( (x_k) \subseteq \mathbb{R}_{S_n} \) converges to \( x \in \mathbb{R}_{S_n} \) in \( \mathcal{T}_n \) if:
\[
\forall \varepsilon \in \mathbb{Q}_{S_n}^+, \exists N \in \mathbb{N}, \forall k \geq N,\quad |x_k - x| < \varepsilon,
\]
and the function \( k \mapsto x_k \) is definable in \( \mathcal{F}_n \).
\end{axiom}

\begin{axiom}[Effective Hausdorff Separation]
\label{ax:hausdorff}
The topology \( \mathcal{T}_n \) is Hausdorff, and the separation is effective \cite{Weihrauch2000,PourElRichards1989}: 
for any distinct \( x, y \in \mathbb{R}_{S_n} \), there exist disjoint basic open sets \( (a, b), (c, d) \in \mathcal{T}_n \) such that:
\begin{itemize}
    \item \( x \in (a, b) \), \( y \in (c, d) \),
    \item The endpoints \( a, b, c, d \in \mathbb{Q}_{S_n} \) are computable in \( \mathcal{F}_n \) from \( x \) and \( y \).
\end{itemize}
\end{axiom}

\begin{theorem}[Effective Minimum on Effectively Compact Sets]
\label{thm:min-attainment}
Let \( K \subseteq \mathbb{R}_{S_n} \) be a non-empty, effectively compact set, and let \( f \colon K \to \mathbb{R}_{S_n} \) be effectively continuous in the sense of Definition~\ref{def:effective-continuity-compact}.

Then there exists a point \( x^* \in K \) such that
\[
f(x^*) = \min_{x \in K} f(x),
\]
(following the constructive minimum principle of \cite{Bishop1967,Bridges1986}) and this minimizer \( x^* \) is effectively approximable in \( \mathcal{F}_n \), i.e., there exists a total \( \mathcal{F}_n \)-definable function \( g \colon \mathbb{N} \to \mathbb{Q}_{S_n} \) such that \( |x^* - g(k)| < 2^{-k} \) for all \( k \in \mathbb{N} \).
\end{theorem}

\begin{proof}[Sketch of Proof]
Since \( K \) is effectively compact, there exists a total \( \mathcal{F}_n \)-definable function \( h \colon \mathbb{N} \to \mathbb{Q}_{S_n} \times \mathbb{Q}_{S_n} \) producing finite covers of \( K \) by open intervals of any desired radius \( \varepsilon > 0 \).

By effective continuity of \( f \), there exists a computable modulus \( \delta(\varepsilon) \in \mathbb{Q}_{S_n}^+ \) such that:
\[
|x - y| < \delta(\varepsilon) \Rightarrow |f(x) - f(y)| < \varepsilon.
\]

Using these facts, we construct a minimizing sequence as follows:

\begin{enumerate}
\item For each \( k \in \mathbb{N} \), compute a finite \( \delta_k \)-net \( \{x_{k,i}\}_{i=1}^{N_k} \subseteq K \) with mesh \( \delta_k := \delta(2^{-k}) \), using the effective compactness of \( K \).
\item Evaluate \( f(x_{k,i}) \) for each \( i \), and select \( x_k \in \{x_{k,i}\} \) such that \( f(x_k) \) is within \( 2^{-k} \) of the minimal value over the net.
\item Define \( x^* := \lim_{k \to \infty} x_k \). Since the modulus ensures \( f(x_k) \to \inf f(K) \), and \( K \) is compact and closed under such limits, \( x^* \in K \).
\end{enumerate}

The function \( k \mapsto x_k \) can be chosen to produce rational approximants \( g(k) \in \mathbb{Q}_{S_n} \) such that \( |x^* - g(k)| < 2^{-k} \), with \( g \) definable in \( \mathcal{F}_n \).

\end{proof}

\section{Axioms of Fractal Arithmetic and Calculus}
\label{sec:fractal-axioms}

\begin{axiom}[Closure under Arithmetic Operations]
\label{ax:closure-arithmetic}
For each level \( n \in \mathbb{N} \), the set \( \mathbb{R}_{S_n} \) is a subfield of \( \mathbb{R} \), satisfying:
\[
0, 1 \in \mathbb{R}_{S_n}, \quad \forall x, y \in \mathbb{R}_{S_n}\ (x \pm y,\ x \cdot y,\ -x \in \mathbb{R}_{S_n}), \quad y \ne 0 \Rightarrow y^{-1} \in \mathbb{R}_{S_n}.
\]
Moreover, if \( x = \lim f(k) \) and \( y = \lim g(k) \), where \( f, g \colon \mathbb{N} \to \mathbb{Q}_{S_n} \) are \( \mathcal{F}_n \)-definable, then \( x \pm y \), \( x \cdot y \), and \( 1/y \) (if \( y \neq 0 \)) admit \( \mathcal{F}_n \)-definable approximations constructed pointwise from \( f \) and \( g \).
\end{axiom}

\begin{axiom}[Definable Integer Arithmetic]
\label{ax:integer-arithmetic}
The set of $S_n$-definable integers is:
\[
\mathbb{Z}_{S_n} := \{ z \in \mathbb{Z} \mid z \text{ is definable in } \mathcal{F}_n \},
\]
and the operations \( + \), \( - \), and \( \times \) on \( \mathbb{Z}_{S_n} \) are computable in \( \mathcal{F}_n \).
\end{axiom}

\begin{axiom}[Effective Approximation]
\label{ax:approximation}
Every real number \( x \in \mathbb{R}_{S_n} \) admits a total sequence \( f \colon \mathbb{N} \to \mathbb{Q}_{S_n} \), computable in \( \mathcal{F}_n \), such that:
\[
\forall k \in \mathbb{N},\quad |x - f(k)| < 2^{-k}.
\]
\end{axiom}

\begin{axiom}[Algebraic Real Inclusion]
\label{ax:algebraic-closure}
Let \( p(x) = a_0 + a_1 x + \dots + a_d x^d \in (\mathbb{Q}_{S_n})[x] \) be a polynomial with coefficients computable in \( \mathcal{F}_n \). If \( x \in \mathbb{R} \) is a computable root of \( p(x) \) in \( \mathcal{F}_n \), then \( x \in \mathbb{R}_{S_n} \).
\end{axiom}

\begin{axiom}[Hierarchy Strictness]
\label{ax:hierarchy-strictness}
For each \( n \in \mathbb{N} \), there exists a real number \( x \in \mathbb{R}_{S_{n+1}} \setminus \mathbb{R}_{S_n} \), i.e., the hierarchy is strictly increasing:
\[
\mathbb{R}_{S_0} \subsetneq \mathbb{R}_{S_1} \subsetneq \dots \subsetneq \mathbb{R}_{S_n} \subsetneq \mathbb{R}_{S_{n+1}} \subsetneq \dots
\]
\end{axiom}

\begin{axiom}[Fractal Differentiability]
\label{ax:diff}
Let \( f \colon \mathbb{R}_{S_n} \to \mathbb{R}_{S_n} \) be a function definable in \( \mathcal{F}_n \). Then \( f \) is differentiable at \( x \in \mathbb{R}_{S_n} \) if there exists \( f'(x) \in \mathbb{R}_{S_{n+1}} \) such that:
\[
\lim_{\substack{h \to 0 \\ h \in \mathbb{R}_{S_n}}} \frac{f(x+h) - f(x)}{h} = f'(x),
\]
where the limit is taken over \( h \in \mathbb{Q}_{S_n} \setminus \{0\} \), with a computable modulus of convergence \( \delta(\varepsilon) \) in \( \mathcal{F}_{n+1} \).
\end{axiom}

\begin{definition}[Fractally Smooth Function]
\label{def:fractally-smooth}
A function \( f \colon \mathbb{R}_{S_n} \to \mathbb{R}_{S_n} \) is said to be \( C^k \)-smooth at level \( n \) if all derivatives up to order \( k \), defined recursively via Axiom~\ref{ax:diff}, exist and are \( \mathcal{F}_{n+k} \)-definable.
\end{definition}

\begin{axiom}[Fractal Integration]
\label{ax:integral}
Let \( f \colon \mathbb{R}_{S_n} \to \mathbb{R}_{S_n} \) be \( \mathcal{F}_n \)-definable and let \( [a, b] \subset \mathbb{R}_{S_n} \). Then the definite integral
\[
\int_a^b f(x)\, dx := \lim_{k \to \infty} \sum_{i=0}^{k-1} f(x_i)\Delta x_i
\]
exists in \( \mathbb{R}_{S_{n+1}} \), where \( x_i = a + i(b-a)/k \in \mathbb{R}_{S_n} \), \( \Delta x_i = (b-a)/k \), and the limit is computable in \( \mathcal{F}_{n+1} \).
\end{axiom}

\begin{axiom}[Calculus Closure]
\label{ax:calculus-closure}
If a function \( f \in \mathcal{F}_n \) is differentiable (resp. integrable) over \( \mathbb{R}_{S_n} \), then its derivative \( f' \) (resp. integral \( \int f \)) is definable in \( \mathcal{F}_{n+1} \).
\end{axiom}

\begin{theorem}[Fundamental Theorem of Fractal Calculus (FTC\(_n\))]
\label{thm:ftc}
Let \( f \colon \mathbb{R}_{S_n} \to \mathbb{R}_{S_n} \) be \( \mathcal{F}_n \)-definable and fractally continuous on \( [a, b] \subset \mathbb{R}_{S_n} \). Define
\[
F(x) := \int_a^x f(t)\, dt
\]
for \( x \in [a, b] \). Then \( F \in \mathcal{F}_{n+1} \), and \( F \) is fractally differentiable with \( F'(x) = f(x) \in \mathbb{R}_{S_{n+1}} \),
in line with constructive formulations of the fundamental theorem of calculus \cite{Bishop1967,Bridges1986,Weihrauch2000}.

Conversely, if \( F \colon \mathbb{R}_{S_n} \to \mathbb{R}_{S_{n+1}} \) is \( \mathcal{F}_{n+1} \)-definable and fractally differentiable with derivative \( F'(x) = f(x) \in \mathbb{R}_{S_n} \), then:
\[
\int_a^b f(x)\,dx = F(b) - F(a).
\]
\end{theorem}

\section{Fractal Number Theory}
\label{sec:number-theory}

\begin{axiom}[Stratified Number Classes]
\label{ax:numbers}
For each definability level \( S_n \) (with \( n \in \mathbb{N} \)), we define:

\begin{itemize}
    \item \textbf{Rational numbers:}
    \[
    \mathbb{Q}_{S_n} = \left\{ \frac{p}{q} \,\middle|\, p, q \in \mathbb{Z}_{S_n},\ q \ne 0 \right\}
    \]
    
    \item \textbf{Algebraic numbers:}
    \[
    \mathbb{A}_{S_n} = \left\{ x \in \mathbb{R}_{S_n} \,\middle|\, \exists p \in (\mathbb{Q}_{S_n})[x] \setminus \{0\} \text{ such that } p(x) = 0 \right\}
    \]
    
    \item \textbf{Transcendental numbers:}
    \[
    \mathbb{T}_{S_n} = \mathbb{R}_{S_{n+1}} \setminus \left( \mathbb{A}_{S_n} \cup \mathbb{Q}_{S_{n+1}} \right)
    \]
    
    \item \textbf{Irrational numbers:}
    \[
    \mathbb{I}_{S_n} = \mathbb{R}_{S_n} \setminus \mathbb{Q}_{S_n}
    \]
\end{itemize}

where \( \mathbb{Z}_{S_n} = \left\{ z \in \mathbb{Z} \,\middle|\, z \text{ is definable in } \mathcal{F}_n \right\} \).
\end{axiom}

\begin{axiom}[Minimal Definability Level]
\label{ax:minimal-level}
Every real number $x \in \mathbb{R}$ has a minimal definability level:
\begin{itemize}
    \item If $x$ is definable, $\exists n_0 \in \mathbb{N}\ \forall n < n_0\ (x \notin S_n) \land x \in S_{n_0}$
    \item If $x$ is not definable in any $S_n$, then $x \notin \mathbb{R}_{S_\omega}$
\end{itemize}
\end{axiom}

\begin{axiom}[Arithmetic Closure of Number Classes]
\label{ax:number-closure}
The classes $\mathbb{Q}_{S_n}, \mathbb{A}_{S_n} \subseteq \mathbb{R}_{S_n}$ are closed under addition and multiplication. Moreover, if $x, y \in \mathbb{A}_{S_n}$, then $x \pm y, xy \in \mathbb{A}_{S_{n+1}}$.
\end{axiom}

\begin{principle}[Computational Hierarchy]
The definability classes satisfy:
\begin{enumerate}
    \item $\mathbb{Q}_{S_n} \subsetneq \mathbb{Q}_{S_{n+1}}$
    \item $\mathbb{A}_{S_n} \subsetneq \mathbb{A}_{S_{n+1}}$
    \item $\mathbb{T}_{S_n} \neq \varnothing$ for all $n \geq 1$
\end{enumerate}
\end{principle}

\begin{principle}[Density]
For every $n \in \mathbb{N}$:
\begin{itemize}
    \item $\mathbb{Q}_{S_n}$ is dense in $\mathbb{R}_{S_n}$ under $\mathcal{T}_n$-topology
    \item $\mathbb{A}_{S_n}$ is dense in $\mathbb{R}_{S_{n+1}}$ under $\mathcal{T}_{n+1}$
\end{itemize}
\end{principle}

\begin{definition}[Fractal Continuum]
The union of all definability levels forms a constructive continuum:
\[
\mathbb{R}_{S_\omega} = \bigcup_{n \in \mathbb{N}} \mathbb{R}_{S_n}
\]
\end{definition}

\begin{conjecture}[Liouville Number Stratification]
\label{conj:liouville}
For Liouville numbers $L$ (as introduced in \cite{Liouville1844} and discussed in stratified form in \cite{Semenov2025FractalCountability}):
\begin{itemize}
    \item All Liouville numbers belong to $\mathbb{R}_{S_\omega}$
    \item For any $n$, there exists $L \in \mathbb{T}_{S_{n+1}} \setminus \mathbb{T}_{S_n}$
    \item No Liouville number belongs to any $\mathbb{A}_{S_n}$
\end{itemize}
\end{conjecture}

\begin{definition}[Stratified Outer Measure]
\label{def:stratified-outer-measure}
Let \( A \subseteq [a,b] \) and fix a definability level \( n \in \mathbb{N} \). The stratified outer measure \( \mu_n^*(A) \) is defined as:
\[
\mu_n^*(A) := \inf \left\{ \sum_{k=1}^N \ell(I_k) \;\middle|\;
\begin{aligned}
&A \subseteq \bigcup_{k=1}^N I_k,\quad I_k = (a_k, b_k) \subseteq [a,b] \\
&a_k, b_k \in \mathbb{Q}_{S_n},\quad \text{and the cover is } \mathcal{F}_n\text{-definable}
\end{aligned}
\right\},
\]
where \( \ell(I_k) = b_k - a_k \) is the length of interval \( I_k \), and \( N \in \mathbb{N} \cup \{\infty\} \).
\end{definition}

\begin{axiom}[Measure-Number Consistency]
\label{ax:measure-number}
Let \( \mu_n^* \) be the stratified outer measure induced by \( \mathcal{F}_n \)-definable covers. Then for all \( n \in \mathbb{N} \):

\begin{itemize}
    \item \textbf{Point Nullity:} For every point \( x \in \mathbb{R}_{S_n} \), the singleton has zero measure:
    \[
    \mu_n^*(\{x\}) = 0.
    \]

    \item \textbf{Transcendental Neighborhood Structure:} If \( x \in \mathbb{T}_{S_n} \), then for every \( \epsilon \in \mathbb{Q}_{S_n}^+ \),
    \[
    \mu_n^*(\mathbb{Q}_{S_n} \cap B_\epsilon(x)) > 0, \quad 
    \mu_n^*(\mathbb{A}_{S_n} \cap B_\epsilon(x)) = 0.
    \]

    \item \textbf{Algebraic Isolation:} If \( x \in \mathbb{A}_{S_n} \setminus \mathbb{Q}_{S_n} \), then \( \exists\, \epsilon > 0 \) in \( \mathbb{Q}_{S_n} \) such that:
    \[
    \mu_n^*(\mathbb{Q}_{S_n} \cap B_\epsilon(x)) = 0.
    \]
\end{itemize}
\end{axiom}

\begin{remark}
This axiom illustrates the fine-grained stratification of number classes under \( \mathcal{F}_n \)-definable measure. Transcendentals exhibit local measure-theoretic richness, while algebraics may become syntactically isolated from rationals in the same layer—mirroring behavior studied in computable measure theory \cite{Weihrauch2000,PourElRichards1989} and further formalized in the stratified setting of \cite{Semenov2025FractalBoundaries}.
\end{remark}

\begin{example}[Stratified Number Classification]
To illustrate the structure of definability levels, we give canonical examples of real numbers and their classification across layers \( S_n \). These examples assume a baseline formal system \( \mathcal{F}_0 \) capable of representing integer arithmetic, basic recursion, and rational operations (e.g., comparable to \( \mathsf{RCA}_0 \)).

\begin{itemize}
    \item \textbf{Level 0}: The system \( \mathcal{F}_0 \) includes all standard integers and rationals with explicitly defined numerators and denominators.
    \[
    0, 1, -2 \in \mathbb{Q}_{S_0},\quad \tfrac{3}{7},\ -\tfrac{11}{5} \in \mathbb{Q}_{S_0}
    \]
    
    \item \textbf{Level 1}: The first definability extension, \( \mathcal{F}_1 \), permits root extraction and solutions to polynomial equations with \( S_0 \)-coefficients, yielding constructively algebraic numbers:
    \[
    \sqrt{2},\ \sqrt[3]{5},\ \text{the golden ratio } \tfrac{1 + \sqrt{5}}{2} \in \mathbb{A}_{S_1}
    \]
    
    \item \textbf{Level 2}: The next layer admits definitions via convergent power series and standard analytic constructions:
    \[
    e = \sum_{k=0}^\infty \frac{1}{k!},\quad \pi = 4 \sum_{k=0}^\infty \frac{(-1)^k}{2k+1} \in \mathbb{T}_{S_2}
    \]
    as studied in classical analysis \cite{Weierstrass1872,Lindemann1882} and revisited in stratified form in \cite{Semenov2025FractalAnalysis}.
    
    \item \textbf{Level 3}: Transcendental numbers with highly nontrivial approximations, such as Liouville numbers, require definability strength beyond that of analytic closure. A classic example:
    \[
    L = \sum_{k=1}^\infty 10^{-k!} \in \mathbb{T}_{S_3} \setminus \mathbb{A}_{S_2}
    \]
    The irrationality measure of \( L \) exceeds any algebraic threshold, making it a natural inhabitant of \( S_3 \) but not definable from lower layers.
\end{itemize}

These examples illustrate how classical number-theoretic classes are stratified according to syntactic complexity and computational definability. Each inclusion reflects the increase in formal resources available at level \( \mathcal{F}_n \).
\end{example}

\begin{example}[Minimal Primitive System]
Assume \( \mathcal{F}_0 \) corresponds to a primitive recursive theory (e.g., Skolem arithmetic), permitting only operations over \( \mathbb{N} \) with basic total functions and no general induction. Then:

\begin{itemize}
    \item \textbf{Level 0}: Only natural numbers and their primitive encodings are definable:
    \[
    0, 1, 2, \dots \in \mathbb{Z}_{S_0},\quad \text{but } \tfrac{1}{2}, -3 \notin \mathbb{Q}_{S_0}
    \]

    \item \textbf{Level 1}: By extending with explicit rational encoding, we obtain:
    \[
    \tfrac{1}{2}, -\tfrac{3}{5}, \tfrac{355}{113} \in \mathbb{Q}_{S_1}
    \]
    
    \item \textbf{Level 2}: Now algebraic operations (roots, polynomial solutions) become expressible:
    \[
    \sqrt{2},\ \sqrt[3]{5},\ \text{roots of } x^2 - x - 1 \in \mathbb{A}_{S_2}
    \]
    
    \item \textbf{Level 3}: Analytic numbers (e.g., via convergent series with rational coefficients):
    \[
    e,\ \pi \in \mathbb{T}_{S_3}
    \]

    \item \textbf{Level 4}: Explicitly constructed Liouville-type numbers, e.g.
    \[
    L = \sum_{k=1}^\infty 2^{-k!} \in \mathbb{T}_{S_4} \setminus \mathbb{A}_{S_3}
    \]
\end{itemize}
This shows how familiar number classes emerge gradually from minimal formal resources.
\end{example}

\begin{example}[Geometrically Generated Fractal Real]
Consider the following construction:
\begin{itemize}
    \item Start with the unit interval \( [0,1] \).
    \item At each step \( k \), remove the open middle \( \tfrac{1}{3} \)-interval from each remaining segment.
    \item Let \( x = \sum_{k=1}^\infty \tfrac{1}{3^{k!}} \).
\end{itemize}

\textbf{Properties}:
\begin{itemize}
    \item Each digit is defined using a simple geometric rule (inclusion/exclusion based on base-3 expansion).
    \item The resulting real number is not algebraic, since it lies in the classical middle-third Cantor set.
    \item It belongs to \( \mathbb{T}_{S_n} \) for some \( n \geq 3 \), depending on how one encodes the limiting process in \( \mathcal{F}_n \).
    \item The base-3 representation of \( x \) avoids the digit 1 entirely, a property verifiable via a definable automaton in higher \( S_n \).
\end{itemize}

This type of number exemplifies how fractal geometry intersects with definability hierarchies: \( x \) is not only non-algebraic, but its description involves a recursively sparse structure with layered computational depth.
\end{example}

\begin{example}[Computable vs Definable Numbers]
This example illustrates the nuanced relationship between computability and definability in the stratified hierarchy. Some real numbers are computable but require higher definability layers to express; others are definable in strong systems but not computable in the classical sense.

\begin{itemize}
    \item \textbf{Computable but not $S_1$-definable:}
    \[
    x = \sum_{k=0}^\infty \frac{1}{2^{2^k}} \in \mathbb{T}_{S_2} \setminus \mathbb{T}_{S_1}
    \]
    This real number is computable via a simple binary expansion algorithm, but the doubly exponential convergence requires resources exceeding $S_1$.

    \item \textbf{Definable but not computable:}
    Chaitin's constant \( \Omega \), which encodes the halting probability of a universal Turing machine \cite{Chaitin1975}, is definable in higher \( S_n \) via oracle-based constructions, but not computable within any effective system.

    \item \textbf{Natural boundary case:}
    \[
    \zeta(3) = \sum_{k=1}^\infty \frac{1}{k^3} \in \mathbb{T}_{S_3}
    \]
    Known to be irrational by Apéry's theorem, yet its precise definability complexity lies between standard arithmetic and transcendental constructions.
\end{itemize}
\end{example}

\begin{example}[Classical but Non-Fractal Reals]
\label{ex:nonconstructive}
This example demonstrates real numbers that are:
\begin{itemize}
    \item Definable in classical set theory (ZFC)
    \item Provably existent but non-constructive
    \item Outside all definability levels $S_n$ in our hierarchy
\end{itemize}

\noindent\textbf{1. Vitali-Type Non-Measurable Supremum}

Let $\{A_i\}_{i \in I}$ be a Vitali partition of $[0,1]$ (existing by the Axiom of Choice). Fix a subset \( I_0 \subset I \) such that both \( I_0 \) and its complement \( I \setminus I_0 \) are non-measurable. Define the real number:
\[
x_V := \sup\left\{r \in \mathbb{Q} \cap [0,1] \mid \exists i \in I_0,\ r \in A_i \right\}.
\]

Defined using a Vitali partition \cite{Vitali1905} of $[0,1]$ under the Axiom of Choice, this supremum exists classically by completeness of \( \mathbb{R} \), but has no computable or constructive description.

\vspace{0.5em}
\noindent\textbf{Definability Analysis:}
\begin{itemize}
    \item \( x_V \notin \mathbb{R}_{S_n} \) for any \( n \), because:
    \begin{enumerate}
        \item The choice of \( I_0 \) requires uncountable and non-effective selection;
        \item No formal system \( \mathcal{F}_n \) can decide whether \( r \in \bigcup_{i \in I_0} A_i \) for a given rational \( r \).
    \end{enumerate}
\end{itemize}

\noindent\textbf{2. Fixed-Point Transcendental Number}
\begin{itemize}
    \item Define the Picard operator $T \in \mathcal{F}_{n+2}$:
    \[
    T(f)(x) = \int_0^x f(t)dt + \frac{x^2}{2}
    \]
    \item The unique fixed point $f_\infty$ satisfies:
    \[
    f_\infty(x) = \sum_{k=0}^\infty \frac{x^{2k+2}}{2^{k+1}(2k+2)!}
    \]
    This example involves a Picard operator used in classical analysis and computability theory \cite{PourElRichards1989}. Evaluation at $x=1$ gives:
    \[
    f_\infty(1) = \sqrt{e} \sinh(1/\sqrt{2}) \approx 0.798
    \]
    
    \textbf{Definability}:
    \begin{itemize}
        \item If $T \in \mathcal{F}_n$, then $f_\infty(1) \in \mathbb{T}_{S_{n+3}}$
        \item Without uniform convergence bounds, $f_\infty(1) \notin \mathbb{R}_{S_\omega}$
        \item The series coefficients are $S_{n+1}$-definable
    \end{itemize}
\end{itemize}

\noindent\textbf{Key Observations}:
\begin{itemize}
    \item \textbf{Measure-theoretic}: $x_V$ lacks Borel/analytic representation
    \item \textbf{Computational}: $f_\infty(1)$ requires arbitrarily high definability
    \item \textbf{Structural}: Both examples satisfy $\mathbb{R}_{S_\omega} \subsetneq \mathbb{R}$
    \item \textbf{Definability-theoretic}: These reals illustrate strict boundaries between classical existence and syntactic realizability
\end{itemize}
\end{example}

\section{Computational Complexity Control}
\label{sec:complexity}

This section formalizes the notion of bounded definability through a hierarchy of formal systems \( \mathcal{F}_n \), each corresponding to a definability layer \( S_n \). These systems govern which functions, sequences, and real numbers are constructively admissible at each level, based on resource-constrained computation.

\begin{axiom}[Bounded Resource Hierarchies]
\label{ax:complexity}
Each definability level \( S_n \) corresponds to a formal system \( \mathcal{F}_n \), defined inductively as follows:

\begin{itemize}
    \item \textbf{Base Level \( \mathcal{F}_0 \)}:
    \begin{itemize}
        \item Time complexity: \( \mathcal{O}(k^c) \), for some fixed \( c \leq 2 \)
        \item Space complexity: \( \mathcal{O}(k) \)
        \item Allowed constructions: elementary arithmetic, bounded loops, finite tables, quantifier-free \( \Delta^0_0 \) schemes
        \item Examples: constant sequences, linear functions, rational evaluation
    \end{itemize}
    
    \item \textbf{Inductive Step \( \mathcal{F}_{n+1} \)}:
    \begin{itemize}
        \item Definitions are permitted if their time complexity satisfies:
        \[
        \operatorname{Time}_{\mathcal{F}_{n+1}}(k) \leq P_n\big( \operatorname{Time}_{\mathcal{F}_n}(k) \big),
        \quad P_n \in \mathrm{Poly}_{\leq 3}
        \]
        \item Here \( P_n(x) = a_0 + a_1 x + a_2 x^2 + a_3 x^3 \), with coefficients bounded by \( a_i \leq 2^{n+2} \)
        \item The language of \( \mathcal{F}_{n+1} \) includes composition of \( \mathcal{F}_n \)-definable functions and \\bounded recursion over \( S_n \)
    \end{itemize}
\end{itemize}
\end{axiom}

\begin{example}[Representative Polynomial Bounds]
Each level \( \mathcal{F}_n \) defines a class of real functions and objects whose approximation or construction is computable within polynomial time bounds, relative to input precision \( k \). Table~\ref{tab:poly-bounds} summarizes representative operations.
\end{example}

\begin{table}[ht]
\centering
\renewcommand{\arraystretch}{1.2}
\begin{tabular}{@{} l l l @{}}
\toprule
\textbf{Level} & \textbf{Time Bound (Worst)} & \textbf{Representative Operations} \\
\midrule
\( \mathcal{F}_0 \) & \( \mathcal{O}(k^2) \) & Integer arithmetic, comparisons \\
\( \mathcal{F}_1 \) & \( \mathcal{O}(k^5) \) & Rational roots, bounded search \\
\( \mathcal{F}_2 \) & \( \mathcal{O}(k^8) \) & $e^x$, $\log x$, rational approximations of $\pi$ \\
\( \mathcal{F}_3 \) & \( \mathcal{O}(k^{12}) \) & Definite integrals, special functions \\
\bottomrule
\end{tabular}
\caption{Sample polynomial complexity bounds by definability level}
\label{tab:poly-bounds}
\end{table}

\begin{theorem}[Controlled Stratification of Complexity]
\label{thm:controlled-growth}
The hierarchy \( \{ \mathcal{F}_n \}_{n \in \mathbb{N}} \) satisfies:
\begin{enumerate}
    \item \textbf{Subexponential growth:} All functions definable in \( \mathcal{F}_n \) are computable in time \( \mathcal{O}(k^{c_n}) \) for some finite constant \( c_n \)
    \item \textbf{Strict inclusion:} \( \mathcal{F}_n \subsetneq \mathcal{F}_{n+1} \) and thus \( S_n \subsetneq S_{n+1} \)
    \item \textbf{Polynomial completeness:} All polynomial-time computable real numbers belong to \( \bigcup_{n} \mathbb{R}_{S_n} \)
\end{enumerate}
\end{theorem}

\begin{definition}[Tamed Transcendentals]
\label{def:tamed-transc}
A real number \( x \in \mathbb{T}_{S_n} \) is called \emph{tamed} if there exists a sequence \( \{q_k\} \subset \mathbb{Q}_{S_n} \) such that:
\[
\forall k \in \mathbb{N},\quad |x - q_k| < 2^{-k}, \quad \text{and } q_k \text{ is computable in time } \mathcal{O}(k^d).
\]
\end{definition}

\paragraph{Interpretation.}
Each level \( \mathcal{F}_n \) encodes not only logical definability, but also bounded constructive access to objects. The stratification reflects a layered semantics of \emph{approximation complexity}, where reals and functions are organized not by ontological status, but by syntactic and resource-based accessibility.

\paragraph{Remark.}
The complexity function \( \rho_n(k) \) may be customized: factorial, exponential, or sublinear growth can be adopted for specific domains (e.g., subrecursive hierarchies, bounded arithmetic). The axioms remain valid under such parameterizations, making the framework adaptable to diverse constructive paradigms.

\section{Principle of Fractal Countability}
\label{sec:fractal-countability}

Despite their layered construction and rich internal structure, the sets of fractal real numbers introduced in this framework remain fundamentally countable. This reflects the syntactic and computational nature of the definability hierarchy.

\begin{theorem}[Fractal Countability]
\label{thm:countability}
The union of all definability levels of fractal real numbers is countable:
\[
\mathbb{R}_{S_\omega} := \bigcup_{n \in \mathbb{N}} \mathbb{R}_{S_n} \quad \text{satisfies} \quad |\mathbb{R}_{S_\omega}| = \aleph_0.
\]
Moreover, \( \mathbb{R}_{S_\omega} \) is equipotent with any other countable set:
\[
\mathbb{R}_{S_\omega} \cong \mathbb{N} \cong \mathbb{Q}.
\]
\end{theorem}

\begin{proof}[Sketch]
Each definability level \( \mathbb{R}_{S_n} \) consists of real numbers defined via \( \mathcal{F}_n \)-definable total functions \( f: \mathbb{N} \to \mathbb{Q}_{S_n} \). Since \( \mathcal{F}_n \) is a formal system with a recursively enumerable set of formulas and syntactic rules, the number of such functions is countable. Taking the union over all \( n \) (a countable index set) yields a countable union of countable sets, which is itself countable.
\end{proof}

\paragraph{Interpretation}
The continuum \( \mathbb{R}_{S_\omega} \) generated by definable approximations is a \emph{fractal continuum}~\cite{Semenov2025FractalBoundaries, Semenov2025FractalCountability}: internally layered, computationally structured, yet globally countable. This avoids classical paradoxes of uncountability and aligns with the philosophy that \emph{definability is the true constraint} in constructive mathematics.

\paragraph{Comparison}
\begin{itemize}
    \item Classical continuum \( \mathbb{R} \): uncountable, contains undefinable reals.
    \item Computable reals: countable, but lack stratified structure.
    \item Fractal reals \( \mathbb{R}_{S_\omega} \): countable, stratified, constructively expressive.
\end{itemize}

\begin{remark}
The countability of \( \mathbb{R}_{S_\omega} \) reflects a central tenet of this framework: \emph{constructive mathematics operates within a countable universe of syntactic descriptions}. While it captures many classical constructions, it does so through layered, definability-aware means.
\end{remark}

\begin{corollary}[Density Preservation]
\label{cor:density}
Although \( \mathbb{R}_{S_\omega} \) is countable, it retains key density properties of the classical real line:
\begin{itemize}
    \item \textbf{Rational Density:} \( \mathbb{Q}_{S_\omega} \) is dense in \( \mathbb{R}_{S_\omega} \) under the topology \( \mathcal{T}_\omega \)~\cite{Semenov2025FractalAnalysis}. That is, for any \( x < y \in \mathbb{R}_{S_\omega} \), there exists \( q \in \mathbb{Q}_{S_\omega} \) such that \( x < q < y \).
    \item \textbf{Algebraic–Transcendental Interpolation:} Between any two distinct points \( x < y \in \mathbb{R}_{S_\omega} \), there exists a number \( z \in \mathbb{A}_{S_\omega} \cup \mathbb{T}_{S_\omega} \) such that \( x < z < y \).
\end{itemize}
\end{corollary}

\section{Constructive Functional Analysis over \texorpdfstring{$\mathbb{R}_{S_n}$}{Rs-n}}
\label{sec:functional-analysis}

This section explores how fundamental results of classical functional analysis can be reformulated within the framework of fractal countability and definability stratification. We focus on a constructive variant of the Hahn--Banach theorem over the layered field of definable real numbers \( \mathbb{R}_{S_n} \).

\subsection{Sublinear Functionals and Extensions}

We begin by recalling the necessary definitions within the definability framework:

\begin{definition}[Sublinear Functional over \( \mathbb{R}_{S_n} \)]
A function \( p : V \to \mathbb{R}_{S_n} \), where \( V \) is a vector space over \( \mathbb{Q}_{S_n} \), is \emph{sublinear} if:
\begin{enumerate}
    \item (Positive homogeneity) \( p(\lambda x) = \lambda p(x) \) for all \( \lambda \in \mathbb{Q}_{S_n}^{+} \), \( x \in V \)
    \item (Subadditivity) \( p(x + y) \le p(x) + p(y) \) for all \( x, y \in V \)
\end{enumerate}
The functional \( p \) is said to be \( \mathcal{F}_n \)-definable if its value on any \( x \in V \) is computable with rational approximations in \( \mathbb{Q}_{S_n} \).
\end{definition}

\begin{definition}[Definable Linear Functional]
Let \( U \subseteq V \) be a subspace over \( \mathbb{Q}_{S_n} \). A function \( f: U \to \mathbb{R}_{S_n} \) is a definable linear functional if it satisfies linearity and is computable in \( \mathcal{F}_n \).
\end{definition}

\subsection{Constructive Extension Theorem}

\begin{theorem}[Constructive Hahn--Banach Theorem over \( \mathbb{R}_{S_n} \)]
\label{thm:hahn-banach}
Let \( V \) be a countable-dimensional \( \mathbb{Q}_{S_n} \)-vector space, and let:
\begin{itemize}
  \item \( f : U \to \mathbb{R}_{S_n} \) be a linear functional defined on a subspace \( U \subseteq V \),
  \item \( p : V \to \mathbb{R}_{S_n} \) be a sublinear functional,
  \item \( f(x) \le p(x) \) for all \( x \in U \),
  \item \( f \) and \( p \) are computable in \( \mathcal{F}_n \).
\end{itemize}
Then there exists a linear extension \( F : V \to \mathbb{R}_{S_{n+1}} \) that:
\begin{enumerate}
    \item Extends \( f \): \( F|_U = f \),
    \item Satisfies \( F(x) \le p(x) \) for all \( x \in V \),
    \item Is computable in \( \mathcal{F}_{n+1} \).
\end{enumerate}
\end{theorem}

\begin{proof}[Sketch]
Since \( V \) is countable-dimensional over \( \mathbb{Q}_{S_n} \), fix an enumeration \( \{v_k\}_{k \in \mathbb{N}} \) of a basis. Define an increasing chain of subspaces \( U_0 \subset U_1 \subset \dots \) such that \( U_0 = U \) and \( U_{k+1} = U_k + \mathbb{Q}_{S_n} \cdot v_k \).

Proceed inductively: at step \( k \), construct a functional \( f_k : U_k \to \mathbb{R}_{S_{n+1}} \) extending \( f_{k-1} \) and preserving \( f_k(x) \le p(x) \). For extension, determine the admissible interval for \( f_k(v_k) \) from the inequality:
\[
    \forall x \in U_k,\quad f_k(x + \lambda v_k) \le p(x + \lambda v_k)
\]
and select a midpoint (or rational approximation) in this interval. This step is effective in \( \mathcal{F}_{n+1} \).

Taking the union \( F = \bigcup_k f_k \), we obtain a total linear functional over \( V \), definable in \( \mathcal{F}_{n+1} \), extending \( f \), and bounded by \( p \).
\end{proof}

\paragraph{Interpretation.}
The classical existential argument is replaced by a constructive procedure of staged extension with definability-controlled approximations. Instead of relying on maximality (via Zorn’s lemma), we perform sequential bounded construction over finite-dimensional stages.

\paragraph{Comparison.}
This version of the Hahn--Banach theorem demonstrates that many results of analysis can be internalized within \( \mathbb{R}_{S_n} \), provided they are appropriately stratified by definability and resource bounds.

\paragraph{Future directions.}
The same methodology can be applied to:
\begin{itemize}
    \item Separation of convex sets with constructively definable hyperplanes,
    \item Definability-sensitive duality in normed spaces over \( \mathbb{R}_{S_n} \),
    \item Quantitative extensions with explicit bounds on approximation complexity.
\end{itemize}

\begin{example}[Sublinear Functional on \( \mathbb{Q}_{S_2}^2 \)]
Let \( V = \mathbb{Q}_{S_2}^2 \) be the 2-dimensional vector space over the rational layer \( \mathbb{Q}_{S_2} \). Define a sublinear functional
\[
p(x, y) := \sqrt{2}_{S_2} \cdot |x| + \pi_{S_2} \cdot |y|,
\]
where \( \sqrt{2}_{S_2}, \pi_{S_2} \in \mathbb{R}_{S_2} \) denote computable approximations of \( \sqrt{2} \) and \( \pi \), respectively, within layer \( S_2 \). That is, both constants are given by total \( \mathcal{F}_2 \)-definable functions \\\( f, g \colon \mathbb{N} \to \mathbb{Q}_{S_2} \) satisfying:
\[
|f(k) - \sqrt{2}| < 2^{-k}, \quad |g(k) - \pi| < 2^{-k}.
\]

Then \( p \) satisfies:

\begin{itemize}
    \item \textbf{Positive homogeneity:} For all \( \lambda \in \mathbb{Q}_{S_2}^+ \),
    \[
    p(\lambda x, \lambda y) = \lambda \cdot p(x, y).
    \]

    \item \textbf{Subadditivity:} For all \( (x_1, y_1), (x_2, y_2) \in \mathbb{Q}_{S_2}^2 \),
    \[
    p(x_1 + x_2, y_1 + y_2) 
    \le \sqrt{2}_{S_2} \cdot (|x_1| + |x_2|) + \pi_{S_2} \cdot (|y_1| + |y_2|) 
    = p(x_1, y_1) + p(x_2, y_2).
    \]
\end{itemize}

Thus, \( p \) is a sublinear functional that is effectively computable in \( \mathcal{F}_2 \).

Now consider the one-dimensional subspace \( U = \mathrm{span}_{\mathbb{Q}_{S_2}}\{(1, 0)\} \subseteq V \), and define the linear functional
\[
f(x, 0) := \sqrt{2}_{S_2} \cdot x, \quad \text{for } x \in \mathbb{Q}_{S_2}.
\]
This functional is linear over \( U \), satisfies \( f \le p \), and is \( \mathcal{F}_2 \)-definable.

By the constructive Hahn–Banach Theorem in the stratified setting, there exists an extension \( \bar{f} \colon \mathbb{Q}_{S_2}^2 \to \mathbb{R}_{S_3} \) such that:

\begin{itemize}
    \item \( \bar{f} \) is \( \mathcal{F}_3 \)-definable,
    \item \( \bar{f} \) is linear over \( V \),
    \item \( \bar{f}(x, y) \le p(x, y) \) for all \( (x, y) \in V \),
    \item \( \bar{f}|_U = f \).
\end{itemize}

\paragraph{Explicit Extension.}
The extension \( \bar{f} \colon \mathbb{Q}_{S_2}^2 \to \mathbb{R}_{S_3} \) can be defined by:
\[
\bar{f}(x, y) := \sqrt{2}_{S_2} \cdot x + \alpha \cdot y,
\]
where \( \alpha \in \mathbb{R}_{S_3} \) is selected (via a definable procedure in \( \mathcal{F}_3 \)) such that:
\[
\forall (x, y) \in \mathbb{Q}_{S_2}^2, \quad \bar{f}(x, y) \le p(x, y) = \sqrt{2}_{S_2} \cdot |x| + \pi_{S_2} \cdot |y|.
\]
This holds whenever \( \alpha \in [-\pi_{S_2}, \pi_{S_2}] \), ensuring sublinearity of the extension. For instance, the choice \( \alpha := 0 \in \mathbb{Q}_{S_3} \) yields a valid extension that satisfies all constraints.

This example illustrates the definability-preserving extension of linear functionals in the context of layered computability, using only syntactic and approximation data available within \( \mathcal{F}_2 \) and its conservative extension \( \mathcal{F}_3 \).
\end{example}

\begin{example}[Stepwise Extension of a Linear Functional]
\label{ex:stepwise-extension}
Let \( V = \mathrm{span}_{\mathbb{Q}_{S_2}} \{1, \sqrt{2}\} \subseteq \mathbb{R}_{S_3} \), and let the subspace \( U = \mathbb{Q}_{S_2} \cdot 1 \cong \mathbb{Q}_{S_2} \). Define the initial linear functional \( f \colon U \to \mathbb{R}_{S_2} \) by \( f(x) = x \).

We aim to construct a \( \mathcal{F}_3 \)-definable extension \( F \colon V \to \mathbb{R}_{S_3} \) such that:
\begin{itemize}
    \item \( F|_U = f \)
    \item \( F(x) \le p(x) \), where \( p \colon V \to \mathbb{R}_{S_2} \) is a sublinear functional defined by:
    \[
    p(x + y\sqrt{2}) := \sqrt{2} \cdot |x| + \sqrt{2} \cdot |y| = \sqrt{2} \cdot (|x| + |y|),
    \]
    which is computable in \( \mathcal{F}_2 \).
\end{itemize}

\textbf{Step 1.} Choose a representative basis element \( v_1 = \sqrt{2} \in V \setminus U \). Compute the extension bounds:
\[
\begin{aligned}
F(\sqrt{2}) &\le p(\sqrt{2}) = 2, \\
-F(-\sqrt{2}) &\le p(\sqrt{2}) = 2.
\end{aligned}
\]
Thus, \( F(\sqrt{2}) \in [-2, 2] \cap \mathbb{R}_{S_3} \). Choose the midpoint:
\[
F(\sqrt{2}) := 1 \in \mathbb{Q}_{S_3}.
\]

\textbf{Step 2.} Define \( F \colon V \to \mathbb{R}_{S_3} \) linearly by:
\[
F(x + y\sqrt{2}) := x + y \cdot 1 = x + y.
\]

\textbf{Step 3.} Verify the constraint \( F(z) \le p(z) \) for all \( z = x + y\sqrt{2} \in V \):
\[
F(x + y\sqrt{2}) = x + y \le \sqrt{2} (|x| + |y|) = p(x + y\sqrt{2}).
\]
This inequality holds for all \( x, y \in \mathbb{Q}_{S_2} \), as \( |x + y| \le |x| + |y| \le \frac{1}{\sqrt{2}} p(x + y\sqrt{2}) \), so \( F \le p \) pointwise.

\textbf{Conclusion.} The functional \( F \colon V \to \mathbb{R}_{S_3} \) is a linear extension of \( f \), definable in \( \mathcal{F}_3 \), and satisfies \( F \le p \). This stepwise method provides a concrete, stratified realization of the Hahn–Banach extension process.
\end{example}

\section{Consistency and Reverse Mathematics}
\label{sec:consistency}

The axiomatic framework developed above is intentionally stratified not only in definability and complexity, but also in logical strength. This section situates the fractal hierarchy \( \{ \mathcal{F}_n \} \) within the landscape of reverse mathematics and ordinal analysis. We relate each definability layer \( S_n \) to standard subsystems of second-order arithmetic and analyze the consistency strength of the overall construction.

\begin{theorem}[Interpretability of Definability Layers]
\label{thm:interpretability}
Each level \( S_n \) of the definability hierarchy is interpretable in second-order arithmetic \( \mathsf{Z}_2 \), with the following correspondences:
\begin{itemize}
    \item \( S_0 \) corresponds to \( \mathsf{RCA}_0 \), capturing recursive constructions.
    \item \( S_1 \) extends to \( \mathsf{ACA}_0 \), admitting arithmetical comprehension.
    \item \( S_2 \) corresponds to \( \mathsf{ATR}_0 \), enabling transfinite recursion on well-orders.
    \item The full union \( S_\omega \) approximates \( \Pi^1_1\text{-}\mathsf{CA}_0 \) within a countable model.
\end{itemize}
\end{theorem}

\begin{theorem}[Proof-Theoretic Strength of the Fractal System]
Let \( \mathcal{F} := \bigcup_{n \in \mathbb{N}} \mathcal{F}_n \) denote the full stratified system. Then:
\begin{itemize}
    \item \( \mathcal{F}_0 \) is conservative over \( \mathsf{PRA} \) (Primitive Recursive Arithmetic) for \( \Pi^0_2 \) statements.
    \item \( \mathcal{F}_\omega \) proves \( \mathrm{Con}(\mathsf{PA}) \) and reaches the strength of \( \mathsf{ACA}_0^+ \).
    \item The full hierarchy \( \{ \mathcal{F}_n \} \) extends to the ordinal strength of \( \mathsf{ID}_1 \) (non-iterated inductive definitions).
\end{itemize}
\end{theorem}

\begin{table}[ht]
\centering
\begin{tabular}{l l l}
\toprule
\textbf{Layer} & \textbf{Reverse Math Equivalent} & \textbf{Ordinal} \\
\midrule
\( S_0 \) & \( \mathsf{RCA}_0 \) & \( \omega^\omega \) \\
\( S_1 \) & \( \mathsf{ACA}_0 \) & \( \varepsilon_0 \) \\
\( S_2 \) & \( \mathsf{ATR}_0 \) & \( \Gamma_0 \) \\
\( S_\omega \) & \( \Pi^1_1\text{-}\mathsf{CA}_0 \) & \( \psi_0(\Omega_\omega) \) \\
\bottomrule
\end{tabular}
\caption{Proof-theoretic correspondence between definability levels and subsystems of arithmetic}
\end{table}

\begin{remark}
The ordinal values listed in the table correspond to the proof-theoretic ordinals associated with the respective subsystems of second-order arithmetic. These ordinals reflect the strength of induction and recursion principles available in each layer:

\begin{itemize}
    \item \( \omega^\omega \) corresponds to primitive recursive arithmetic as formalized in \( \mathsf{RCA}_0 \),
    \item \( \varepsilon_0 \) arises from \( \mathsf{ACA}_0 \), which includes arithmetical comprehension,
    \item \( \Gamma_0 \) is the ordinal of \( \mathsf{ATR}_0 \), supporting arithmetical transfinite recursion,
    \item \( \psi_0(\Omega_\omega) \) is the Bachmann–Howard ordinal, capturing the strength of full \( \Pi^1_1 \) comprehension as in \( \mathsf{CA}_0 \).
\end{itemize}

These ordinals emerge from formal ordinal analysis and serve as canonical invariants for measuring the logical and constructive strength of each definability layer \( S_n \).
\end{remark}

\paragraph{Beyond Second-Order Arithmetic.}

While the stratified framework is interpretable in \( \mathsf{Z}_2 \) and recovers major subsystems of second-order arithmetic such as \( \mathsf{RCA}_0 \), \( \mathsf{ACA}_0 \), and \( \mathsf{ATR}_0 \), this correspondence should not be mistaken for a limitation. The ability to reconstruct fragments of \( \mathsf{Z}_2 \) is a baseline consistency check, not the central goal of the system.

Crucially, the framework enables the construction of arbitrary computable arithmetic universes with tunable logical and topological properties. As shown in Section~\ref{sec:complexity}, the definability levels \( S_n \) can be parameterized by customized resource bounds (e.g., subexponential, factorial, or even non-uniform). This means the framework does not merely simulate classical arithmetic — it generates entire families of stratified arithmetics, each with their own effective continuum.

One can, for instance, define a real number system in which only dyadic rationals or Cantor-like numbers with base-3 expansions of the form \( 0.\overline{0}, 0.2, 0.02, \dots \) are allowed. Once such a definability filter is fixed, all axioms and theorems of fractal analysis remain valid, including notions of topology, continuity, integration, and differentiation over this synthetic number system.

This flexibility allows the user to integrate over fractal domains (such as the middle-third Cantor set), analyze the behavior of functions over sparse or constrained spectra of reals, or model alternative continua without losing logical consistency. The framework thus provides a robust, constructive meta-language in which different flavors of analysis, number theory, and topology can be instantiated in a controlled, definability-aware way.

\section{Applications}
\label{sec:applications}

\subsection{Approximation Theory}

Classical approximation theorems such as the Jackson estimate assert that continuous functions can be uniformly approximated by polynomials with error bounded by a modulus of continuity. However, these theorems are typically non-constructive: they guarantee existence without providing a means of effective approximation~\cite{Bishop1967, Bridges1986}. 

In the stratified framework developed here, this limitation is overcome. Approximation becomes a \emph{layer-relative} process: both the target function and the approximating polynomial are definable within a bounded formal system \( \mathcal{F}_n \), and the modulus of continuity is computable within the same layer. This gives rise to a constructively meaningful and computationally verifiable version of Jackson-type estimates.

\begin{definition}[\( S_n \)-Uniform Norm]
Let \( f, g \colon \mathbb{R}_{S_n} \to \mathbb{R}_{S_n} \) be \( \mathcal{F}_n \)-definable functions. The \( S_n \)-uniform norm (or supremum norm) of their difference is defined as:
\[
\|f - g\|_\infty := \sup_{x \in [a,b] \cap \mathbb{R}_{S_n}} |f(x) - g(x)|,
\]
where the interval \( [a,b] \subset \mathbb{R}_{S_n} \) is \( \mathcal{F}_n \)-definable and compact in the topology \( \mathcal{T}_n \). The supremum is taken constructively, i.e., via an \( \mathcal{F}_n \)-definable approximation process.
\end{definition}

\begin{theorem}[Layer-Wise Jackson-Type Estimate]
Let \( f \in C(S_n) \) be a function continuous with respect to the topology \( \mathcal{T}_n \). Then there exists a polynomial \( p_n \in \mathbb{Q}_{S_n}[x] \) such that:
\[
\|f - p_n\|_\infty \leq C_n \cdot \omega_f(\Delta_n),
\]
where \( \omega_f \) is the modulus of continuity of \( f \), \( \Delta_n \) is an \( \mathcal{F}_n \)-definable approximation step, and \( C_n \) is a constant depending on the topology \( \mathcal{T}_n \).
\end{theorem}

\subsection{Computable Analysis}

Classical computable analysis, especially in the Type-2 Effectivity (TTE) framework, formalizes real-number computation via oracle machines operating on fast-converging rational sequences. While precise, this model does not capture internal stratifications of definability or resource bounds. The stratified framework introduced here refines computability by embedding it within a layered hierarchy of formal systems.

In this setting, functions are not only computable, but \emph{stratified computable}—their action on real arguments respects definability levels and complexity constraints. This enables constructive control over approximation depth, computational cost, and definitional transparency.

\begin{definition}[Stratified Type-2 Computability]
A function \( f \colon \mathbb{R}_{S_n} \to \mathbb{R}_{S_m} \) is called \emph{stratified computable} if:
\begin{itemize}
    \item It maps every \( \mathcal{F}_n \)-definable Cauchy sequence \( (x_k) \) with limit in \( \mathbb{R}_{S_n} \) to a Cauchy sequence \( (f(x_k)) \) converging in \( \mathbb{R}_{S_m} \);
    \item The transformation \( x_k \mapsto f(x_k) \) is definable in \( \mathcal{F}_{\max(n,m)} \).
\end{itemize}
\end{definition}

\subsection{Algorithmic Mathematics}

Many classical problems in theoretical computer science — including optimization, verification, and approximation — require not just the construction of solutions, but the identification of syntactic or semantic gaps that separate approximate solutions from exact ones. In the classical PCP theorem, such gaps are verified through probabilistic local checks~\cite{Chaitin1975}. In the stratified framework of fractal definability, we observe a similar phenomenon: even when a problem is entirely definable within a level \( S_n \), the construction of gap witnesses or separating certificates typically requires moving to level \( S_{n+1} \).

This reflects a general syntactic asymmetry: the existence of a solution and the verification of its optimality (or suboptimality) may not reside in the same definability layer.

\begin{definition}[Fractal PCP Gap Principle]
Let \( \mathcal{P} \) be an optimization or decision problem whose objective function, constraints, and feasible set are all definable in \( \mathcal{F}_n \). We say that \( \mathcal{P} \) satisfies the \emph{Fractal PCP Gap Principle} if any minimal constructive witness of a nonzero approximation gap lies in:
\[
S_{n+1} \setminus S_n.
\]
\end{definition}

\begin{theorem}[Fractal PCP Gap Theorem]
\label{thm:pcp-gap}
Let \( \mathcal{P} \) be an optimization problem defined over \( \mathbb{R}_{S_n}^d \), with \( \mathcal{F}_n \)-definable objective function \( f \colon \mathbb{R}_{S_n}^d \to \mathbb{R}_{S_n} \) and feasible region \( D \subseteq \mathbb{R}_{S_n}^d \). Suppose \( x^* \in D \) is a suboptimal point such that \( f(x^*) + \delta < \inf f(D) \) for some \( \delta \in \mathbb{Q}_{S_n}^+ \). Then there exists a syntactic certificate \( w \in S_{n+1} \setminus S_n \) verifying the existence of this gap:
\[
f(x^*) + \delta < f(y), \quad \text{for all } y \in D, \text{ witnessed via } w.
\]
\end{theorem}

\begin{remark}
This ascent from \( S_n \) to \( S_{n+1} \) reflects the increase in definitional complexity required to syntactically isolate error bounds, separating hyperplanes, or infeasibility proofs. The principle provides a constructive analogue of classical duality gaps and hardness-of-approximation boundaries, but within a finitely stratified framework.
\end{remark}

\begin{remark}[On the Role of the Witness \( w \)]
In Theorem~\ref{thm:pcp-gap}, the expression ``witnessed via \( w \in S_{n+1} \setminus S_n \)'' refers to the fact that the existence of an approximation gap cannot be established within \( \mathcal{F}_n \), but becomes syntactically provable or verifiable in \( \mathcal{F}_{n+1} \). 

Formally, \( w \) may represent:
\begin{itemize}
    \item a concrete real number, inequality, or bound (e.g., a rational interval or separation constant) not definable in \( \mathcal{F}_n \), but constructible in \( \mathcal{F}_{n+1} \);
    \item a derivation or proof tree (e.g., in bounded arithmetic or a formal system) that confirms the suboptimality of a given solution;
    \item an effective separator (e.g., dual certificate, Lagrangian multiplier, polynomial lower bound) definable only at level \( S_{n+1} \).
\end{itemize}

Thus, \( w \) is not just a numerical value but a syntactic object — a definable expression or procedure — that encodes the minimal constructive evidence of a nonzero approximation gap. Its non-membership in \( S_n \) signifies that the existence of the gap cannot be validated within the original definability layer.
\end{remark}

\begin{example}[Fractal Gap Certificate in a Simple Optimization Problem]
Let \( f(x) = x^2 - 2x + 1 \), defined over the interval \( [0,1] \cap \mathbb{Q}_{S_2} \). The minimum is attained at \( x = 1 \), where \( f(1) = 0 \). Consider the candidate point \( x_0 = 0.9 \in \mathbb{Q}_{S_2} \). Then:
\[
f(x_0) = (0.9)^2 - 2(0.9) + 1 = 0.81 - 1.8 + 1 = 0.01.
\]

To confirm suboptimality of \( x_0 \), it is not sufficient to compare \( f(x_0) \) with the unknown minimum value; instead, we must show that:
\[
\forall y \in [0,1] \cap \mathbb{Q}_{S_2}, \quad f(y) \geq f(x_0).
\]
However, such a universal comparison requires evaluation of \( f(y) \) for arbitrary \( y \in \mathbb{Q}_{S_2} \) with sufficient precision.

\paragraph{Need for \( \mathcal{F}_3 \)-witness.}
To construct a witness \( w \in S_3 \setminus S_2 \) confirming suboptimality of \( x_0 \), we require:
\[
\exists \delta > 0 \in \mathbb{Q}_{S_3} \quad \text{such that} \quad \forall y \in \mathbb{Q}_{S_2} \cap [0,1], \quad f(y) \geq 0.009.
\]
This lower bound (strictly below \( f(x_0) = 0.01 \)) cannot be verified within \( \mathcal{F}_2 \) due to limited approximation power — e.g., \( \mathcal{F}_2 \) cannot uniformly resolve values of \( f(y) \) to within \( 10^{-3} \) precision. The witness \( w = 0.009 \in \mathbb{Q}_{S_3} \) serves as a certificate of the gap.

\paragraph{Conclusion.}
The existence of such \( w \) illustrates the Fractal PCP Gap Principle: although the function \( f \) and candidate point \( x_0 \) are definable in \( \mathcal{F}_2 \), the syntactic verification of their suboptimality requires ascent to \( \mathcal{F}_3 \).
\end{example}

\section{Conclusion}
\label{sec:conclusion}

This work introduces a stratified axiomatic framework for real analysis and number theory grounded in computable definability. The approach constructs a hierarchy of formal systems \( \mathcal{F}_n \), each generating a definability layer \( \mathbb{R}_{S_n} \subseteq \mathbb{R} \), culminating in the union \( \mathbb{R}_{S_\omega} \)—a constructive, countable, and topologically rich continuum.

\subsection*{Summary of Contributions}

\begin{itemize}
    \item \textbf{Fractal Definability Framework}:
    \begin{itemize}
        \item Formal axiomatization of countably stratified systems \( \mathcal{F}_n \), inducing definability layers \( S_n \).
        \item Introduction of fractal topology \( \mathcal{T}_n \), arithmetic closure, calculus, and measure at each level.
        \item Consistency with and interpretability in known systems such as \( \mathsf{RCA}_0 \), \( \mathsf{ACA}_0 \), and beyond.
    \end{itemize}

    \item \textbf{Layered Number Classes}:
    \[
    \mathbb{Q}_{S_n} \subsetneq \mathbb{A}_{S_n} \subsetneq \mathbb{T}_{S_n} \subsetneq \mathbb{R}_{S_{n+1}}
    \]
    Each layer distinguishes rational, algebraic, transcendental, and non-definable real numbers by their computational accessibility. The total union \( \mathbb{R}_{S_\omega} \) forms the fractal continuum.

    \item \textbf{Constructive Universality}:
    The framework does not fix a particular number system, but rather \emph{generates families of constructive continua} by parametrizing definability, complexity, and convergence. As shown in Section~\ref{sec:complexity}, it supports the construction of arbitrary computable number systems—e.g., real numbers based on Cantor-set encodings, bounded approximations, or transfinite schemes. All axioms and theorems remain valid over any such construction, provided it is expressible within some \( \mathcal{F}_n \)~\cite{Semenov2025FractalAnalysis}.
\end{itemize}

\subsection*{Boundary and Non-Constructivity}

\begin{theorem}[Definability Boundary Theorem]
Let \( x \in \mathbb{R} \setminus \mathbb{R}_{S_\omega} \). Then:
\[
\forall n \in \mathbb{N},\quad x \notin \mathbb{R}_{S_n}, \quad \text{yet } \mathsf{ZFC} \vdash \exists x.
\]
\end{theorem}

Non-constructive reals—such as those arising from Vitali sets~\cite{Vitali1905}, non-measurable selectors~\cite{BanachTarski1924}, or Chaitin's constant~\cite{Chaitin1975}—act as horizons for finitely grounded mathematics.

\begin{itemize}
    \item \textbf{Foundational Role:} They delimit the boundaries of provability, continuity, and definability.
    \item \textbf{Computational Interpretation:} Their absence in \( \mathbb{R}_{S_\omega} \) corresponds to the inaccessibility of exact computation or approximation.
    \item \textbf{Physical Insight:} Measurable quantities in physical systems are plausibly always representable in \( \mathbb{R}_{S_n} \) for some finite \( n \).
\end{itemize}

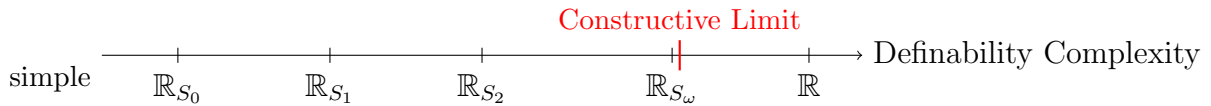
\begin{center}
\begin{tikzpicture}
\draw[->] (0,0) -- (10,0) node[right] {Definability Complexity};
\node[below left] at (0,0) {\small simple};
\foreach \x/\l in {1/$\mathbb{R}_{S_0}$,3/$\mathbb{R}_{S_1}$,5/$\mathbb{R}_{S_2}$,7.5/$\mathbb{R}_{S_\omega}$,9.3/$\mathbb{R}$}
    \draw (\x,0.1) -- (\x,-0.1) node[below] {\l};
\draw[red, thick] (7.6,-0.2) -- (7.6,0.2) node[above] {\small Constructive Limit};
\end{tikzpicture}
\end{center}

\paragraph{Final Reflection.}
The fractal continuum \( \mathbb{R}_{S_\omega} \) is not a limitation of classical analysis, but rather its computational refinement. Within this universe:
\begin{itemize}
    \item Proofs are syntactic derivations and algorithms;
    \item Theorems correspond to verifiable constructions;
    \item Continuity, integrability, and differentiability are grounded in definability;
    \item New constructive analogues of classical results—such as Hahn–Banach, \\Stone–Weierstrass, or Tietze—can be formulated layer by layer.
\end{itemize}

This framework serves as both a unifying foundation for constructive mathematics and a generative tool for building new mathematical worlds grounded in definitional precision and computational realism.


\end{document}